\documentclass[10pt]{amsart}
\usepackage{amssymb}
\usepackage{amsfonts}
\usepackage{amsthm}
\usepackage{mathtools}
\usepackage{todonotes}

\calclayout

\usepackage{amsthm}

\usepackage[pagebackref,colorlinks,citecolor=blue,linkcolor=blue]{hyperref}

\usepackage{cite}   

\usepackage{amsmath,bbm,amssymb,amsxtra}

\def\Xint#1{\mathchoice
	{\XXint\displaystyle\textstyle{#1}}
	{\XXint\textstyle\scriptstyle{#1}}
	{\XXint\scriptstyle\scriptscriptstyle{#1}}
	{\XXint\scriptscriptstyle\scriptscriptstyle{#1}}
	\!\int}

\def\XXint#1#2#3{{\setbox0=\hbox{$#1{#2#3}{\int}$ }
		\vcenter{\hbox{$#2#3$ }}\kern-.6\wd0}}
\def\dashint{\Xint-}

\newtheorem{theorem}{Theorem}[section]
\newtheorem{lemma}[theorem]{Lemma}

\newtheorem{corollary}[theorem]{Corollary}
\newtheorem{definition}[theorem]{Definition}
\newtheorem{example}[theorem]{Example}

\numberwithin{equation}{section}

\title{Existence and uniqueness of limits at infinity for homogeneous Sobolev functions}

\author{Pekka Koskela}
\address{Department of Mathematics and Statistics \\
P.O. Box 35 \\
FI-40014 University of Jyväskylä, Finland}
\email{pekka.j.koskela@jyu.fi}

\author{Khanh Nguyen}
\address{ Faculty of Mathematics, Mechanics and Informatics,
			 VNU University of Science, Vietnam National University, Hanoi, Vietnam}
\email{khanhnn@vnu.edu.vn}


\begin{document}
\maketitle
\begin{abstract}
  We establish the existence and uniqueness of  limits at infinity along infinite curves outside a zero modulus family for functions in a homogeneous Sobolev space under the assumption that the underlying space is equipped with a doubling measure which supports a Poincar\'e inequality.  We also  characterize the settings where this conclusion is nontrivial.
Secondly, we introduce notions of weak polar coordinate systems and radial curves on metric measure spaces. Then sufficient and necessary conditions for existence of radial limits are given. As a consequence, we   characterize  the existence of radial limits in certain concrete settings.
\end{abstract}

\section{Introduction}
Let $(X,d,\mu)$ be a metric measure space with metric $d$ and Borel regular measure $\mu$. A locally rectifiable curve $\gamma:[0,\infty)\to X$ is an \textit{infinite curve}  if   $\gamma\setminus B\neq\emptyset$ for all balls $B$. Then $\int_{\gamma}ds=\infty$ if $\gamma$ is an infinite curve. We write $\Gamma^\infty$ for the collection of all infinite curves and denote by $\dot N^{1,p}(X), 1\leq p<\infty,$ the collection of all locally integrable functions that have a $p$-integrable upper gradient on $(X,d,\mu)$. Here the notion of upper gradients is given in Section \ref{modulus-capacity}.

The aim of this paper is to study the existence and uniqueness of the limit 
	\begin{equation}
\label{infinity-limit}\lim_{t\to\infty}u(\gamma(t))
\end{equation} for $\gamma\in\Gamma^\infty$ and for $u\in\dot N^{1,p}(X)$.
We say that \textit{ the existence and uniqueness} of \eqref{infinity-limit} hold for $p$-a.e  $\gamma\in\Gamma^\infty$ if,  for every   $u\in\dot N^{1,p}(X)$, there exists  $c\in \mathbb R$  such that  
\[\lim_{t\to\infty} u(\gamma(t)) \text{ exists and }\lim_{t\to\infty} u(\gamma(t))= c \text{\rm \ \  for $p$-a.e $\gamma\in\Gamma^\infty$}.
\]  Here the notion of $p$-a.e curve is given in Section \ref{modulus-capacity}.

Towards uniqueness, we employ an annular chain property at a given based point $O$, see  Definition \ref{dyadic}. This chain property holds, for example, if  there exists a constant $C\geq 1$ for which any pair of points in $B(O,r)\setminus B(O,r/2)$ can be joined by a curve in $B(O,C r)\setminus B(O,r/C),$ see Section \ref{sec-dyadic} for more details. This holds especially when $X$ is annularly quasiconvex as defined in \cite[Section 8.3]{pekka}.

Our first result deals with the limit \eqref{infinity-limit} for $p$-a.e $\gamma\in\Gamma^\infty$.

\begin{theorem} 
	\label{thm1} Let $1\leq p<\infty$. Suppose that $(X,d,\mu)$ is a doubling metric measure space that supports a $p$-Poincar\'e inequality. Assume that $X$ has the annular chain property.
	Then the existence and uniqueness of \eqref{infinity-limit} hold for $p$-a.e $\gamma\in\Gamma^\infty$.
\end{theorem}

In Theorem \ref{thm1}, the uniqueness does not hold without  some additional assumption besides doubling and Poincar\'e whose definitions are given in Section \ref{admissible}.  For example, on a space with at least two ends one easily constructs a Lipschitz function in $\dot N^{1,p}(X)$ so that the tail value of two of the ends is a different constant. A tree in \cite{BBGS17,NgWa20}  (or a weighted real line in \cite{BBK06,BBS20}) is doubling, supports a Poincar\'e inequality and has more than one end.

The conclusion of Theorem \ref{thm1} is nontrivial only when the $p$-modulus of $\Gamma^\infty$ is strictly positive. 
Let us suppose that $\mu$ is a $Q$-Ahlfors regular measure as in Section \ref{admissible} where $1< Q\leq p<\infty$. Then the $p$-modulus of $\Gamma^\infty$ vanishes, and hence there exists $u\in\dot N^{1,p}(X)$ such that $\lim_{t\to\infty}u(\gamma(t))=\infty$ for every $\gamma\in\Gamma^\infty$, see \cite[Page 135]{pekka}.  One can actually characterize the settings where the $p$-modulus of the family $\Gamma^\infty$ is strictly positive. This is the content of Theorem \ref{thm1.2-1802} below.
	
Let $O$ be a fixed point in $X$ and let $A_{2^j}(O)$ be the annuli $B(O,2^{j+1})\setminus B(O,2^j)$ for $j\in\mathbb N$. We define 
\[  \mathcal R_{p}(O):=\sum_{j\in \mathbb N}(2^j)^{\frac{p}{p-1}}\mu^{\frac{1}{1-p}}(A_{2^{j}}(O)) \text{\rm \ \ if $p>1$ and \ \ } \mathcal R_{1}(O):=\sup_{j\in\mathbb N} 2^j\mu^{-1}(A_{2^j}(O)).\]
The finiteness of $\mathcal R_{p}(O)$ is then a volume growth condition. A reader familiar with classification theory should recognize this as a condition towards $p$-hyperbolicity  \cite{HoKo01,Hol99,G99}.
	
\begin{theorem}\label{thm1.2-1802}
Let $1\leq p<\infty$. Suppose that $(X,d,\mu)$ is a complete doubling metric measure space that supports a $p$-Poincar\'e inequality. Assume that $X$ has the annular chain property at O. Then the following statements are equivalent:
\begin{enumerate}
\item[I.] $\mathcal R_p(O)<\infty$.
\item[II.]$\text{\rm Mod}_p(\Gamma^\infty)>0$.
\item[III.] For every $u\in \dot N^{1,p}(X)$, there exists an infinite curve $\gamma\in\Gamma^\infty$ such that $\lim_{t\to\infty}u(\gamma(t))$ exists.
\end{enumerate}
\end{theorem}

The space $X=[0,\infty)\subset \mathbb R$ equipped with the Lebesgue measure and the usual distance is doubling and supports a $p$-Poincar\'e inequality for all $p\ge 1.$ 
It also has the annular chain property and there is (modulo reparametrizations and translations) only one injective infinite curve. Moreover, $\mathcal R_p(0)<\infty$ precisely when $p=1,$ in which case every
function in $\dot N^{1,1}(X)$ has a (unique) limit along this curve. Hence the existence of a single ``good" curve is the best one can obtain in this generality. However, for 
$\mathbb R^n,$ $n\ge 2,$ equipped with the Lebesgue measure and Euclidean distance, $\mathcal R_p(0)<\infty$ precisely when $p<n$ and one then obtains a unique limit along
the radial half-line in the direction $\xi\in S^{n-1},$ for almost every $\xi$ with respect to the surface measure on the boundary $S^{n-1}$ of the unit ball.  This is a consequence of the existence of spherical (or polar) coordinates. Towards establishing limits along a recognizable family of infinite curves, let us introduce an abstract version of polar coordinates.

Let $\mathbb S$ be a nonempty set (a set of indices), $O\in X$ and consider collections $\Gamma^{O}(\mathbb S)$ consisting of
$\gamma_{\xi}\in \Gamma^{\infty},$  $\xi\in \mathbb S,$  with 
$\gamma_{\xi}(0)=O.$ 
We say that $X$ has a \textit{weak polar coordinate system} at the \textit{coordinate point $O$} if there exist a set $\mathbb S$ of indices with a Radon probability measure $\sigma$ on 
$\mathbb S$, a choice of $\Gamma^{O}(\mathbb S),$ a function $h:X\to\mathbb [0,\infty)$, and a constant $\mathcal C>0$   such that 
\begin{equation}
\label{intro-polar}
\int_{\mathbb S}\int_{\gamma^{O}_\xi}|f|h\ dsd\sigma\leq \mathcal C\int_{X}|f|d\mu \text{\rm \ \ for every integrable function  $f$. }
\end{equation}
We then call $h$ a
\textit{coordinate weight} and each $\gamma^{O}_\xi\in\Gamma^{O}(\mathbb S)$ a \textit{radial curve} in the direction $\xi\in\mathbb S$  from $O$. 
Towards the existence of  \eqref{infinity-limit} along radial curves, consider a  polar coordinate \eqref{intro-polar} at  $O$.
 We set
\[
R_{p}(h,O):=\int_{X\setminus B(O,1)}h^{\frac{p}{1-p}}d\mu \text{\ \  if $p>1$ and\ \ }R_{1}(h,O):=\| h^{-1}\|_{L^\infty_\mu(X\setminus B(O,1))}.
\]
Given $\xi\in F\subset \mathbb S$ with $\liminf_{t\to\infty}d(O,\gamma_\xi^O(t))>1$, there is $t_\xi>0$ such that $d(O,\gamma_\xi^O(t_\xi))=1$ and $d(O,\gamma_\xi^O(t))>1$ for all $t>t_\xi$. Let $\hat\gamma_\xi^O:[0,\infty)\to X$ be the infinite curve starting from $\gamma_\xi^O(t_\xi)$ defined by $\hat\gamma_\xi^O(t):=\gamma_\xi^O(t+t_\xi)$ for $t\geq 0$. 
The collection of all these infinite curves $\hat\gamma_\xi^O$ with respect to $\xi\in F$ satisfying $\liminf_{t\to\infty}d(O,\gamma_\xi^O(t))>1$ is denoted by $\hat\Gamma^O(F)$.

In order to state our next result, we introduce the following properties:
\begin{enumerate}
	\item[1.] For every $u\in \dot N^{1,p}(X)$, there is  a radial curve $\gamma$ such that $\liminf_{t\to\infty}u(\gamma(t))$ is finite.
	\item[2.] The limit \eqref{infinity-limit} exists for $\sigma$-a.e $\xi \in \mathbb S$, i.e. the limit
	\[\lim_{t\to\infty}u(\gamma^{O}_\xi(t)) \text{\rm \ \ exists for $\sigma$-a.e $\xi\in\mathbb S$, for every $u\in\dot N^{1,p}(X)$}.
	\]
	\item[3.]   The existence and uniqueness of \eqref{infinity-limit} are obtained for $\sigma$-a.e $\xi\in\mathbb S$, i.e.  for every  $u\in\dot N^{1,p}(X),$ there exists  $c\in\mathbb R$  such that 
	\[\lim_{t\to\infty}u(\gamma^{O}_\xi(t))\text{\rm \ exists\ and \ }\lim_{t\to\infty}u(\gamma^{O}_\xi(t))=c \text{\rm \ \  for $\sigma$-a.e $\xi\in\mathbb S$}.
	\]
	\item[4.]  $\text{\rm Mod}_p(\hat\Gamma^{O}(\mathbb S))>0$. 
	\item[5.] If $F\subseteq \mathbb S$ satisfies $\sigma(F)>0,$ then $\text{\rm Mod}_p(\hat\Gamma^{O}(F))>0.$ 
\end{enumerate}

We show that each of  $\{1., 2., 3., 4.,5.\}$ is ``between" $\mathcal R_p(O)<\infty$ and $R_p(h,O)<\infty.$
\begin{theorem}
	\label{thm4} Let $1\leq p<\infty$. Suppose that $(X,d,\mu)$ is a doubling metric measure space that supports a $p$-Poincar\'e inequality. Assume that $X$ has the annular chain property at O.
	Suppose that $X$ has  a weak polar coordinate system $(\mathbb S,\sigma,\Gamma^{O},h)$ at $O$ as in \eqref{intro-polar}.
Then 
\begin{enumerate}
\item[I.] The condition $R_p(h,O)<\infty$ is sufficient for all of $\{ 1., 2., 3.,4.,5. \} $.
\item[II.] The condition $\mathcal R_p(O)<\infty$  is necessary for each of $\{ 1., 2., 3.,4.,5. \} $.
\end{enumerate}
 \end{theorem}

Theorem \ref{thm4} leads us to compare the finiteness of $R_p(h,O)$ and   $\mathcal R_p(O).$  First of all, it immediately follows from 
Theorem \ref{thm4} that the finiteness
of $R_p(h,O)$ guarantees the finiteness of $\mathcal R_p(O).$ We do not know a simple direct proof for this.  On the other hand, the finiteness of $\mathcal R_p(O)$ does not in general yield the finiteness of $R_p(h,O).$ To see this, simply consider $\mathbb R^n,$ $n\geq 2,$ with the Euclidean distance and Lebesgue measure, usual spherical coordinates (with normalized measure on $S^{n-1}$) but with the coordinate weight $h=\chi_{\mathbb R^n\setminus B(0,1)}.$ On the other hand, for example, if
	\begin{equation}\label{eq-1802}\notag
h(x)\gtrsim\sum_{j\in\mathbb N}\frac{\mu(A_{2^j}(O))}{(2^j)^p}\chi_{A_{2^j}(O)}(x)
	\end{equation}
 for all $x\in X,$ then $R_p(h,O)\lesssim\mathcal R_p(O).$ 

Let us next consider the  Muckenhoupt $\mathcal A_p$-weighted space $(\mathbb R^n,d_E,w)$, where $n\geq 2$ and $O$ is the origin. Here $d_E$ is the Euclidean distance, and $w$ is a Muckenhoupt $\mathcal A_p$-weight with the associate measure $d\mu=wdx$, i.e. there is a constant $C\geq 1$ such that  for every ball $B\subset\mathbb R^n$,
\begin{equation}\label{eq1.6-0610}
\left(\dashint_{B}wdx\right)^{\frac{1}{p}}\left(\dashint_{B}w^{\frac{1}{1-p}}dx\right)^{\frac{p}{p-1}}\leq C \text{\rm \ \ if $p>1$,}
\end{equation}
and
\begin{equation}\label{eq1.7-0610}
\left(\dashint_{B}wdx \right)\|w^{-1}\|_{L^\infty(B)}\leq C\text{\rm \ \ if $p=1$}.
\end{equation}
Then $\mu$ is doubling and supports a $p$-Poincar\'e inequality, see for instance \cite{HKM06}. The annular chain property at $O$ is satisfied and the usual (normalized) spherical coordinate system satisfies \eqref{intro-polar} at $O.$  We show that $R_{p}(h,O)\approx \mathcal R_{p}(O)$ in Example \ref{lem4.3}. In particular, $R_{p}(h,O)<\infty$ if and only if $1\leq  p<n$ for the unweighted space $\mathbb R^n$. It follows that Theorem \ref{thm4} recovers some of the conclusions in \cite{EKN,Uspen}. 

Next, we consider a $Q$-Ahlfors regular space $X$ that supports a $p$-Poincar\'e inequality, where $1\leq p<\infty$ and  $1\leq Q<\infty.$ Suppose that  $X$ has the annular chain property at $O$ and satisfies the strong form 
\[
\int_{X\setminus B(O,1)}|f|d\mu=\int_{\mathbb S}\int_1^\infty |f(\gamma_\xi(t))|t^{Q-1}\ dtd\sigma(\xi) \text{\rm \ \ (for every integrable function $f$)}
\]
of our weak polar coordinate system.
 Then 
each of $\{1., 2., 3., 4.,5.\}$ is equivalent to $1\leq p<Q$ if $1<Q<\infty$, and  equivalent to  $p=1$ if $Q=1$  since 
\begin{equation}
\label{carnot}\mathcal R_{p}(O)<\infty \implies \begin{cases}
1\leq p<Q & \text{\rm \ \ if \ \ }1< Q<\infty,\\
p=1 & \text{\rm \ \ if \ \ }Q=1.
\end{cases}
\implies R_{p}(h,O)<\infty
\end{equation}
where 
\[R_{p}(h,O)\approx \begin{cases}
\int_1^\infty r^{\frac{(Q-1)}{1-p}}dr &\text{\rm if \ \ }p>1,\\
\|r^{1-Q}\|_{L^\infty([1 ,\infty))} &\text{\rm if \ \ }p=1,
\end{cases}
\text{\ \ and\ \ } \mathcal R_{p}(O)\approx \begin{cases}
\sum_{j\in\mathbb N}(2^j)^{\frac{p-Q}{p-1}}&\text{\rm if \ \ }p>1,\\
\sup_{j\in\mathbb N}(2^j)^{1-Q}&\text{\rm if \ \ }p=1.
\end{cases}
\]
In particular,  by \cite{BT02,Ty20} every polarizable  Carnot group $G$ (especially every group of Heisenberg type) admits a polar coordinate system  \eqref{carnot-polar} that satisfies \eqref{intro-polar}.  Since the Haar 
measure on $G$ is  
$Q$-Ahlfors regular for the homogeneous dimension $Q>1$ of $G$ and supports a $p$-Poincar\'e inequality for each $p$,  each of  $\{1., 2., 3., 4.,5.\}$ is equivalent to $1\leq p<Q<\infty$.

We then obtain the following characterization.

\begin{corollary}\label{local-2908} Let $1\leq p<\infty$ and let $1<Q<\infty$. Then
	\begin{enumerate}
		\item[I.] Each of $\{1., 2., 3., 4.,5.\}$ is equivalent to  $R_{p}(h,O)<\infty$ (or $\mathcal R_{p}(O)<\infty$)   on  the
		Muckenhoupt $\mathcal A_p$-weighted Euclidean spaces.
		\item[II.]  Each of $\{1., 2., 3., 4.,5.\}$ is equivalent to $1\leq p<Q<\infty$  on  polarizable Carnot groups. 
	\end{enumerate}
\end{corollary}

The organization of this paper is as follows. In Section \ref{sec2}, we introduce polar coordinates and recall the basic notions on metric measure spaces. In Sections \ref{sec3}-\ref{sec4}, proofs of Theorem \ref{thm1}-\ref{thm1.2-1802}-\ref{thm4} are given.

Throughout this paper, we use the following conventions. We denote by $O$  the base point in the annular chain property and also refer by $O$  to the coordinate 
point in the polar coordinate \eqref{intro-polar}. 
The notation $A\lesssim B\ (A\gtrsim B)$ means that there is a constant $C>0$ only depending on the data such that $A\leq C \cdot B\ (A\geq C\cdot B)$, and $A\approx B$ means that $A\lesssim B$ and $A\gtrsim B$. For  each locally integrable function $f$ and for every measurable subset $A\subset X$ of positive measure, we let $f_A:=\dashint_Afd\mu=\frac{1}{\mu(A)}\int_Afd\mu$.
\section{Polar coordinates and preliminaries}\label{sec2}
 In Section \ref{sec-polar}, we introduce polar coordinates. We  recall the basic notions of  modulus, doubling, Poincar\'e inequalities, and Lebesgue points in Sections \ref{modulus-capacity}-\ref{admissible}-\ref{lebesgue}.
\subsection{Polar coordinates}\label{sec-polar}
Let $(X,d)$ be a metric space. A  \textit{curve} is a nonconstant continuous mapping from an interval $I\subseteq\mathbb R$ into $X$. The \textit{length} of a curve $\gamma$ is denoted by $l(\gamma)$. A curve $\gamma$ is said to be a  \textit{rectifiable curve} if its length is finite. Similarly, $\gamma$ is a   \textit{locally rectifiable curve} if its restriction to each compact subinterval of $I$ is rectifiable. Each rectifiable curve $\gamma$ will be parameterized by arc length and hence the  \textit{line integral} over $\gamma$ of a Borel function $f$ on $X$ is 
\[\int_\gamma fds =\int_0^{l(\gamma)}f(\gamma(t))dt.
\]
If $\gamma$ is locally rectifiable, then we set 
\[\int_{\gamma}fds=\sup\int_{\gamma'}fds
\]where the supremum is taken over all rectifiable subcurves $\gamma'$ of $\gamma$. Let $\gamma:[0,\infty)\to X$ be a locally rectifiable curve, parameterized by arc length. Then
\[\int_{\gamma}fds=\int_0^\infty f(\gamma(t))dt.
\]
A locally rectifiable curve $\gamma$ is an  \textit{infinite curve} if $\gamma\setminus B\neq\emptyset$ for all balls $B$. Then $\int_{\gamma}ds=\infty.$   
 We denote by $\Gamma^\infty$ the collection of all infinite curves.

Let $\mathbb S$ be a nonempty set (a set of indices). Given a point $O$, we consider collections $\Gamma^{O}(\mathbb S)$ of infinite curves with parameter space $\mathbb S$ starting from $O$, namely
\[\Gamma^{O}(\mathbb S)=\{\gamma^{O}_\xi\in\Gamma^\infty: \gamma_\xi^{O}(0)=O, \xi\in\mathbb S\}.
\]
We say that $(X,d,\mu)$  has a   \textit{weak polar coordinate system} at the \textit{coordinate point} $O$ if there is a choice of a pair $(\mathbb S, \Gamma^{O}(\mathbb S))$  with a  Radon probability measure $\sigma$ on $\mathbb S$, a \textit{coordinate weight} $h:X\to\mathbb [0,\infty)$, and a constant $\mathcal C>0$ such that 
\begin{equation}\label{polar-0103}
\int_{\mathbb S}\int_{\gamma^{O}_\xi}|f|\ h\ dsd\sigma\leq\mathcal C \int_{X}|f|d\mu \text{\rm \ \ for every integrable function $f$.}
\end{equation}
Each $\gamma_\xi^{O}\in\Gamma^{O}(\mathbb S)$ is called a   \textit{radial curve} with respect to $\xi\in\mathbb S$ (starting from $O$). Notice that we have not assumed $h$ to be strictly positive. In applications it will be important to require $h$ to be strictly positive at least almost everywhere in the union of the images of $\gamma_\xi^{O}.$

Let us give examples of weak polar coordinates. First of all, 
in the $n$-dimensional Euclidean space $\mathbb R^n$, where $n\geq 2$, one has the usual spherical coordinate system at a given point $O$: there is a constant $C(n)>0$ depending on $n$ such that
\begin{equation}
\int_{\mathbb R^n}f(x)dx=C(n)\int_{ S^{n-1}}\int_0^\infty f(O+r\cdot\xi)r^{n-1}drd\mathcal H^{n-1} \label{usual-uniform-polar}\text{\rm \ \ for  every integrable function $f$}. 
\end{equation}Here $\mathcal H^{n-1}$ is the (normalized) $(n-1)$-Hausdorff measure and $S^{n-1}$ is the unit sphere centered at $O$. 
For these coordinates, \eqref{polar-0103} holds as an identity.

In \cite{BT02,Ty20}, Balogh and Tyson  produced a  polar coordinate system at the origin on polarizable Carnot groups $G$:  there is a family of horizontal curves $\gamma_\xi:[0,\infty)\to G$ where $\xi$ ranges over a certain compact unit sphere $S\subset G$, and a positive Radon measure $\sigma$ on $S$ so that 
\begin{equation}
\label{carnot-polar} \int_{G}f(g)dg=\int_S\int_0^\infty f(\gamma_\xi(t))t^{Q-1}dtd\sigma(\xi)
\end{equation}
 is valid for every integrable function $f$. Here the integral on the left is taken with respect to the Haar measure in the group and $Q$ denotes the homogeneous dimension of $G$. The unit sphere $S$ is the level set $\{g\in G: N(g)=1\}$ for a certain homogeneous norm $N$ in $G$. 

In \cite[Lemma 3.1]{KNW}, a polar coordinate at the root $O$ on a $K$-regular tree is given. Let $X$ be a $K$-regular tree with its boundary $\partial X.$  We equip  $X$ with a radially weighted distance $\lambda$ and a radially weighted measure $\mu$. Then there is a uniform measure $\nu$ on $\partial X$ such that 
\begin{equation}\notag
\label{tree-polar} \int_{X}fd\mu\approx\int_{\partial X}\int_{[O,\xi)}f(x) \frac{K^{|x|}\mu(x)}{\lambda(x)}ds(x) d\nu(\xi)  
\end{equation}for every integrable function $f$. 
 Here $[O,\xi):=\gamma_\xi$ is the unique geodesic ray from the root $O$ to $\xi\in\partial X$.
 
 In the above examples, one actually has two-sided estimates for the terms in \eqref{polar-0103}. This is not necessarily the case for our weak polar coordinate systems. For example, let $X:=\{(x_1,x_2)\in\mathbb R^2: -1\leq x_2\leq 1\}\cup \{(x_1,x_2)\in\mathbb R^2: |x_2|\leq x_1\}$ and equip it with the Lebesgue measure and the Euclidean distance. Let $S_1=[-1,1]$ and set $\gamma_\xi^1(t)=\begin{cases} (-t,t\cdot\xi),& 0\leq t\leq 1 \\ (-t,\xi),& t>1 \end{cases},$ when $\xi\in[-1,1]$. Define $h_1(x_1,x_2)=\begin{cases}1,& x_1<-1\\ 0,&{\rm otherwise.} \end{cases}$ Then 
 \[
 \int_{S_1}\int_{0}^\infty|f(\gamma_{\xi}^1(t))|\ h_1(\gamma_\xi^1(t))\ dt d\xi \leq \int_X|f(x_1,x_2)| dx_1 dx_2 \text{\rm \ \ for every integrable function $f$}.
 \]This is a weak polar coordinate system. A second weak polar coordinate system on $X$ is obtained by taking $S_2=[-\pi/4, \pi/4]$, $\gamma_\xi^2(t)=e^{i\cdot\xi}$ and $h_2(x_1,x_2)=\|(x_1,x_2)\|\chi_{\{(x_1,x_2)\in X: x_2\geq 0\}}(x_1,x_2)$. Then 
 \[
 \int_{S_2}\int_{0}^\infty|f(\gamma_{\xi}^2(t))|\ h_2(\gamma_\xi^2(t))\ dt d\xi \leq \int_X|f(x_1,x_2)| dx_1 dx_2 \text{\rm \ \ for every integrable function $f$}.
 \]
 For a third one, we take $S_3=S_1\cup S_2, h_3=h_1+h_2$ and $\gamma_\xi^3=\gamma_\xi^1\cup\gamma_\xi^2$. Then
 \[
 \int_{S_3}\int_{0}^\infty|f(\gamma_{\xi}^3(t))|\ h_3(\gamma_\xi^3(t))\ dt d\xi \leq \int_X|f(x_1,x_2)| dx_1 dx_2 \text{\rm \ \ for every integrable function $f$}.
 \]

In \cite[Section 3]{KM98}, there is a weak polar coordinate system on the Cantor $\infty$-diamond. Let us describe it. Let $E_i$ be the usual Cantor set in the unit interval $[i,i+1]$ obtained by first taking out the middle interval of length $1/3$ and leaving two intervals of length $1/3$ and then continuing inductively. The Cantor $\infty$-diamond, denoted by $X$, is obtained by replacing each of the complementary intervals of $E_i$ by a square having that interval as one of its diagonals. Thus we have a line of diamonds along the unit interval, and they are joined up by $E_i$. We consider the map $F:[0,\infty)\times[-1,1]\to X$ defined by 
	\[
		F(x,y)=\left(x,\delta(x)\tan \left(\pi\frac{y}{4}\right)\right).
	\]
Here $\delta(x)$ is the distance from $x$ to $E:=\bigcup_{i=0}^\infty E_i$. Then the map $F$ is simply the vertical projection on $E\times[-1,1]$ and it is one-to-one, locally bi-Lipschitz on $([0,\infty)\setminus E)\times[-1,1]$. The Jacobian of $F$ at $(x,y)$, denoted by $J_F(x,y)$, is $\frac{\pi\delta(x)}{4\cos^2(\pi y/4)}$. Let $\Gamma$ be the family of curves $\gamma_y$, $-1<y<1$, defined by  
	\[
		\gamma_y(x)=F(x,y) {\rm \ \ for\ all\ }y\in[-1,1].
	\]  
Then $\gamma_y$ is $4$-Lipschitz for each $y\in[-1,1]$. We denote by $\sigma$ the $1$-Lebesgue measure on [-1,1]. By the change of variables formula, we have 
	\begin{align*}
		\int_{\Gamma}\int_{\gamma}f\cdot (J_F\circ F^{-1})dsd\left(\frac{\sigma}{2}\right)(\gamma)=&\frac{1}{2}\int_{[-1,1]}\left(\int_{[0,\infty)} (f\circ\gamma_y)(x)\cdot (J_F\circ F^{-1}\circ \gamma_y)(x)\cdot |\dot\gamma_y(x)|dx\right)dy\\ 
		\leq& 2\int_{[-1,1]}\int_{[0,\infty)} (f\circ F)(x,y)\cdot J_F(x,y)dxdy \\
		=&2\int_X f d\mathcal L^2
	\end{align*}
for each positive integrable function $f$ on $X$ with respect to the $2$-Lebesgue measure $\mathcal L^2$. Here the inequality is obtained since $\gamma_y$ is $4$-Lipschitz for each $y\in[-1,1]$ and $(J_F\circ F^{-1}\circ\gamma_y)(x)=(J_F\circ F^{-1})(F(x,y))=J_F(x,y)$.

Finally, we give a definition of Semmes-type families which are related to our weak polar coordinate system.
\begin{definition}[A Semmes type family of infinite curves] \label{definition2.1}
Let $(X,d,\mu)$ be a metric measure space with metric $d$ and measure $\mu$. Given $O\in X$, a family $\Gamma$ of infinite curves starting from $O$ is called Semmes-type if there exist a constant $C>0$ and a Radon probability measure $\sigma$ on $\Gamma$ such that 
\begin{equation}\label{eq-semmes}
\int_{\Gamma}\int_{\gamma}f\ ds\ d\sigma(\gamma)\leq C \int_{X}f(x) \frac{d(x,O)}{\mu(B(O,d(x,O)))}d\mu(x)
\end{equation} 
for every positive measurable function $f$ on $X$ for which the right-hand side of \eqref{eq-semmes} is finite.
\end{definition}

For instance, a family of all radial curves is a Semmes type family on $\mathbb R^n$ by the usual spherical coordinate system \eqref{usual-uniform-polar} or on a polarizable Carnot group by the coordinate system \eqref{carnot-polar}. This definition is naturally modified from the existence of families of rectifiable curves joining pairs of points on metric measure spaces by Semmes in \cite{semmes}.
\subsection{Modulus}\label{modulus-capacity}
 Let $\Gamma$ be a family of curves in a metric measure space $(X,d,\mu)$. Given $1\leq p<\infty$, the  \textit{$p$-modulus} of $\Gamma$, denoted  $\text{\rm Mod}_p(\Gamma)$, is defined by 
\[\text{\rm Mod}_p(\Gamma):=\inf \int_{X}\rho^pd\mu
\]where the infimum is taken over all Borel functions $\rho:X\to[0,\infty]$ satisfying 
$\int_{\gamma}\rho ds\geq 1
$ for every locally rectifiable curve $\gamma\in\Gamma$. A family of curves is called \textit{$p$-exceptional} if it has $p$-modulus zero. We say that a property  holds for \textit{$p$-a.e curve} if the collection of curves for which the property fails  is $p$-exceptional.

Let $u$ be a locally integrable function. A Borel function $\rho:X\to[0,\infty]$ is said to be an \textit{upper gradient} of $u$ if 
\begin{equation}
\label{def-upper-gradient}|u(x)-u(y)|\leq \int_{\gamma}\rho ds
\end{equation}
for every rectifiable curve $\gamma$ connecting $x$ and $y$. Then we  have that \eqref{def-upper-gradient} holds for all compact subcurves of $\gamma\in\Gamma^\infty$.
We say that $\rho$ is a  \textit{$p$-weak upper gradient} of $u$ if \eqref{def-upper-gradient} holds for $p$-a.e rectifiable curve. In what follows, we denote by $g_u$ the  \textit{minimal upper gradient} of $u$, which is unique up to measure zero and which is minimal in the sense that $g_u\leq\rho $ a.e. for every $p$-integrable $p$-weak upper gradient $\rho$ of $u$. In \cite{H03}, the existence and uniqueness of such a minimal upper gradient are given.

The notion of upper gradients is due to Heinonen and Koskela \cite{HK98}, we refer  interested  readers to \cite{BB15,H03,HK98,N00} for a more detailed discussion on upper gradients.

\subsection{Doubling  and Poincar\'e inequalities}\label{admissible}
Let $(X,d)$ be a metric space. A Borel regular measure $\mu$ is called \textit{doubling} if every ball in $X$ has  finite positive  measure and if there exists a constant $C\geq 1$ such that for all balls $B(x,r)$  with radius $r>0$ and center at $x\in X$,
\[\mu(B(x,2r))\leq C\mu(B(x,r)).
\]
Let $1\leq Q<\infty$. A Borel regular measure $\mu$ is said to be \textit{$Q$-Ahlfors regular} if there exists a constant $C\geq 1$ such that for all balls $B(x,r)$ with radius $r>0$ and center at $x\in X$,
\[\frac{r^Q}{C}\leq \mu(B(x,r))\leq Cr^Q.
\] Hence if $\mu$ is $Q$-Ahlfors regular for some $1\leq Q<\infty$, then $\mu$ is a doubling measure.

Let $1\leq p<\infty$. We say that a measure $\mu$ supports a  \textit{$p$-Poincar\'e inequality} if every ball in $X$ has  finite positive measure and if there exist constants $C>0$ and $\lambda\geq 1$ such that 
\[\dashint_{B(x,r)}|u-u_{B(x,r)}|d\mu \leq C r\left (\dashint_{\lambda\cdot B(x,r)}\rho^pd\mu\right )^{\frac{1}{p}}
\] for all balls $B(x,r)$ with radius $r>0$ and center at $x\in X$, and for all pairs $(u,\rho)$ satisfying \eqref{def-upper-gradient} such that $u$ is integrable on balls. Here $\lambda\cdot B(x,r):=B(x,\lambda\cdot r)$ and  $\lambda$ is called the scaling constant of $p$-Poincar\'e inequality or the scaling factor of $p$-Poincar\'e inequality. Since $u\in \dot N^{1,p}(X)$ is integrable on balls, the $p$-Poincar\'e inequality makes sense for any pair $(u,\rho_u)$ where $\rho_u$ is an upper gradient of $u$.

\subsection{Lebesgue points}\label{lebesgue}
 A point $x\in X$ is called a Lebesgue's point of $u$ if $\lim_{r\to0} \dashint_{B(x,r)}|u(y)-u(x)|d\mu(y)=0$ where $B(x,r)$ is the ball with radius $r$ and center at $x$. Let $N_u$ be a set of all points $x\in X$ such that $x$ is not a Lebesgue point of $u$.
 \begin{theorem}[Lebesgue differentiation theorem, see for instance {\cite[Page 77]{pekka}}]\label{thm2.3-2809}
 We have $\mu(N_u)=0$ for every locally integrable function $u$ on $X$.
 \end{theorem}
 
\subsection{Chain conditions}
\label{sec-dyadic}
In this paper, we employ the following  annular chain property. 
\begin{definition}
	\label{dyadic}
	Let $\lambda\geq1$. We say that $X$ satisfies an annular $\lambda$-chain condition at $O$  if the following holds. There are constants $c_1\geq 1, c_2\geq 1, \delta>0$ and a finite number $M<\infty$  so that given $r>0$ and points $x,y\in B(O,r)\setminus B(O,r/2)$, one can find balls $B_1, B_2, \ldots, B_k$ with the following properties:
	\begin{enumerate}
		\item[1.] $k\leq M$.
		\item[2.] $B_1=B(x,r/(\lambda c_1))$, $B_k=B(y,r/(\lambda c_1))$ and the radius of  each $B_i$ is $r/(\lambda c_1)$ for $1\leq i\leq k$. 
		\item[3.] $ B_i\subset B(O,c_2r)\setminus B(O,r/c_2)$ for $1\leq i\leq k$.
		\item[4.] For each $1\leq i\leq k-1$, there is a ball $D_i\subset B_i\cap B_{i+1}$ with radius $\delta r$.
	\end{enumerate}
	If $X$ satisfies an annular $\lambda$-chain condition at $O$ for every $\lambda\geq 1$, we say that $X$ has the annular chain property.
\end{definition}
If $X$ satisfies an annular $\lambda$-chain condition at $O$ for some $\lambda\geq 1$ then it also satisfies an annular $\lambda'$-chain condition at $O$ for all $\lambda'$ with $1\leq \lambda'\leq\lambda$ by taking $c_1(\lambda'):=c_1(\lambda)\frac{\lambda}{\lambda'}\geq 1$. Here $c_1(\lambda)$ is the constant $c_1$ with respect to $\lambda$ as in Definition \ref{dyadic}. It follows that $X$ has the annular chain property if and only if there is a sequence $\lambda_k$ with $\lim_{k\to+\infty}\lambda_k=+\infty$ such that $X$ satisfies an annular $\lambda_k$-chain condition at $O$ for each $k\in\mathbb N$.
\begin{lemma}Let $\mu$ be doubling on $(X,d)$.
Suppose that there is a constant $c_0\geq 1$ so that for every $r>0$, each pair $x,y$ of points in $B(O,r)\setminus B(O,r/2)$ can be joined by a curve in $B(O,c_0r)\setminus B(O,r/c_0)$. Then $X$ has the annular chain property.
\end{lemma}
\begin{proof}
We follow ideas in the proof of \cite[Theorem 7.2]{HK00}. Let $\lambda\geq 1$. Let $\gamma_{x,y}$ be a curve  in $B(O,c_0r)\setminus B(O,r/c_0)$ joining $x,y\in B(O,r)\setminus B(O,r/2)$.  We consider the collection of all balls $B(w,r/(100\lambda))$ with $w\in B(O,c_0r)\setminus B(O,r/c_0)$. As $\mu$ is doubling, by the $5B$-covering lemma, we find a cover of  $\gamma_{x,y}$ consisting of $k$ of these balls, say $D_1,D_2,\ldots,D_k,$ with $k$ depending only on $c_0,\lambda$ and the doubling constant, so that the following properties hold:
\begin{enumerate}
\item $\gamma_{x,y}\subset\bigcup_{i=1}^k5D_i$.
\item $\{D_i\}_{i=1}^k$ are pairwise disjoint.
\item $5D_i\bigcap 5D_{i+1}\neq\emptyset$ for $1\leq i\leq k-1$.
\end{enumerate}
Let $B_i:=20 D_i$. Then the four properties as in  Definition \ref{dyadic} can be checked to hold for a subcollection of the balls $B_i$ when $\delta=1/(100\lambda), c_1=1/5, c_2=10c_0$.
\end{proof}

\begin{corollary}\label{cor2.4-0103}
If $X$ is annularly quasiconvex as defined  in \cite[Section 8.3]{pekka}, then $X$ has the annular chain property. 
\end{corollary}
It especially follows from Corollary \ref{cor2.4-0103} together with \cite[Theorem 9.4.1]{pekka} that every complete metric measure space that is $Q$-Ahlfors regular and supports a $p$-Poincar\'e inequality for some $1\leq p<Q$ has the annular chain property.
\section{Proof of Theorem \ref{thm1}}\label{sec3}

In this section, we prove  Theorem \ref{thm1}.
 Throughout this section, we always assume that $(X,d,\mu)$ is doubling and supports a $p$-Poincar\'e inequality with scaling factor $\lambda\geq 1$, and that $u\in\dot N^{1,p}(X)$ with $\rho_u$  a $p$-integrable upper gradient of $u$. We assume that $X$ satisfies an annular $\lambda$-condition at $O$. Let $B(x,r)$ be the ball with radius $r>0$ and center at $x\in X$. We set $\tau B(x,r):=B(x,\tau r)$, $A_{\lambda 2^{j+1}}:=B(O,\lambda2^{j+2})\setminus B(O,\lambda2^{j+1})$, $c_2\cdot A_{\lambda 2^{j+1}}:=B(O,c_2\lambda 2^{j+3})\setminus B(O,2^{j}/(c_2\lambda))$ for  $\tau>0$, $j\in\mathbb N$ where  $c_2$ is as in Definition \ref{dyadic}. Let us begin with useful lemmas, established using ideas from \cite{HK98,HK95,HK00}. 
\begin{lemma}\label{lem3.1} 	Let $1\leq p<\infty$ and $j\in\mathbb N$. Suppose that
	 $E,F$ are two subsets of $A_{\lambda 2^{j+1}}$ such that $|u(x)-u(y)|\geq 1$ for all $x\in E, y\in F$, and that
		$|u(x)-u_{B(x,2^j/c_1)}|\leq {1}/{5}$, $|u(y)-u_{B(y,2^{j}/c_1)}|\leq {1}/{5}$ for some $x\in E$, $y\in F$, where $c_1\geq 1$ is as in Definition \ref{dyadic}.
		Then 
		\begin{equation}\label{eq3.2}
		\mu(c_2\cdot A_{\lambda 2^{j+1}})\lesssim (2^{j})^p \int_{c_2\cdot A_{\lambda 2^{j+1}}}\rho_u^pd\mu.
		\end{equation}
\end{lemma}
\begin{proof}We have that $1\leq |u(x)-u(y)|\leq {1}/{5} + |u_{B(x,2^j/c_1)}-u_{B(y,2^j/c_1)}|+{1}/{5}.$
	Hence
		$
		1\lesssim |u_{B(x,2^j/c_1)}-u_{B(y,2^j/c_1)}|.
		$
		Let  $\{B_i\}_{i=1}^k$ statisfy the four properties  in Definition \ref{dyadic}.
	Since $\mu$ is  doubling and supports a $p$-Poincar\'e inequality, it follows that 
		\begin{align*}
		1&\lesssim |u_{B(x,2^j/c_1)}- u_{B(y,2^j/c_1)}|
		\lesssim \sum_{i=1}^k \dashint_{B_i}|u-u_{B_i}|d\mu \lesssim \sum_{i=1}^k 2^j \left (\dashint_{\lambda B_i}\rho_u^pd\mu\right )^{\frac{1}{p}}.
		\end{align*}
	Hence there is an index $i$  such that 
		$1\lesssim 2^j \left (\dashint_{\lambda B_i}\rho_u^pd\mu\right )^{\frac{1}{p}}.
		$
	By the doubling property, we obtain  that 
		\[\mu(c_2\cdot A_{\lambda 2^{j+1}})\lesssim \mu(\lambda B_i)\lesssim (2^j)^p  \int_{\lambda B_i}\rho_u^pd\mu \leq (2^j)^p  \int_{c_2\cdot A_{\lambda 2^{j+1}}}\rho_u^pd\mu
		\] which is \eqref{eq3.2}. 
\end{proof}
\begin{lemma}\label{lem3.2} Let $1\leq p<\infty$ and $j\in \mathbb N$. Suppose that
	 $E$ is a subset of $A_{\lambda 2^{j+1}}$ such that 
	$|u(x)-u_{B(x,2^j/c_1)}|> {1}/{5}
	$ holds for every $x\in E$, where $c_1\geq 1$ is as in Definition \ref{dyadic}. Then 
	\begin{equation}\label{eq3.5}
	\mu(E)\lesssim (2^{j})^p \int_{c_2\cdot A_{\lambda 2^{j+1}}}\rho_u^pd\mu.
	\end{equation}
\end{lemma}
\begin{proof}Let $x\in E$. Set $B_i=B(x, 2^{j-i}/c_1)$, $i\in\mathbb N$. {{By Theorem \ref{thm2.3-2809}, we may assume that  every $x\in E$ is a Lebesgue point of $u$.}} Since  $\mu$ is doubling and supports a $p$-Poincar\'e inequality, we have that 
		\begin{align}
		\sum_{i=0}^\infty 2^{-i}\approx \frac{1}{5}< &|u(x)-u_{B(x,2^j/c_1)}|\leq \sum_{i=0}^\infty |u_{B_i}-u_{B_{i+1}}|
		 \lesssim \sum_{i=0}^\infty 2^{j-i}\left (\dashint_{\lambda B_i}\rho_u^pd\mu \right )^{\frac{1}{p}}.\label{eq3.5-1107}
		\end{align}
	 Thus there is an index $i_x$ such that 
	 	$ 2^{-i_x}\lesssim2^{j-i_x} \left (\dashint_{\lambda B_{i_x}}\rho_u^pd\mu \right )^{\frac{1}{p}}$ and so
		\begin{equation}
		\label{eq3.8} \mu(\lambda B_{i_x}) \lesssim  (2^{j})^p \int_{\lambda B_{i_x}}\rho_u^pd\mu.
		\end{equation}
		Hence $E$ has a cover $\{\lambda B_{i_x}: \eqref{eq3.8} \text{\rm \ holds}\}_{x\in E}$. Using the $5B$-covering lemma, there is a pairwise disjoint collection $\{ \lambda B_{i_{x_k}}\}_{k=1}^\infty$ such that $E\subset \bigcup_{k=1}^\infty5\lambda B_{i_{x_k}}$.
		By \eqref{eq3.8}, since $\mu$ is doubling, we obtain that
		\[\mu(E)\leq \sum_{k=1}^\infty\mu(5\lambda B_{i_{x_k}})\lesssim \sum_{k=1}^\infty \mu({\lambda B_{i_{{x_k}}}})\lesssim \sum_{k=1}^\infty (2^j)^p\int_{\lambda B_{i_{x_k}}}\rho_u^pd\mu \leq (2^j)^p\int_{c_2\cdot A_{\lambda 2^{j+1}}}\rho_u^pd\mu
		\]which is \eqref{eq3.5}.
	 Here the last inequality of above is given since $\{\lambda B_{i_{x_k}}\}_{k=1}^\infty$ are pairwise disjoint in $c_2\cdot A_{\lambda 2^{j+1}}$. 
	  The claim follows.
\end{proof}
\begin{proof}[Proof of Theorem \ref{thm1}]
	Let
		$A:=\left \{\gamma\in\Gamma^\infty: \int_\gamma \rho_uds=\infty\right \}.
		$
	We then obtain from \cite[Lemma 5.2.8]{pekka} that  $\text{\rm Mod}_p(A)=0$ since  $u\in \dot N^{1,p}(X)$, and hence 
		$\lim_{t\to\infty}u(\gamma(t)) \text{\rm \ \ exists for all $\gamma \in \Gamma^\infty\setminus A$}.
		$
	It remains to show the  uniqueness of $c$ in \eqref{infinity-limit} for $p$-a.e $\gamma\in\Gamma^\infty$.  We argue by contradiction. 
	By adding a suitable constant to $u$ and finally by multiplying $u$ by another suitable constant, we may assume that
	 there exist two subfamilies of $\Gamma^\infty$, denoted $\Gamma_E,\Gamma_F$, and $\delta>0$ such that 
		\begin{equation}\label{eq3.10}
		\text{\rm Mod}_p(\Gamma_E)\geq \delta>0, \text{\rm Mod}_p(\Gamma_F)\geq \delta>0, \lim_{t\to\infty}u(\gamma(t)) \geq 2, \lim_{t\to\infty}u(\gamma'(t))\leq 0 \text{\rm \ \ for all $\gamma\in\Gamma_E,\gamma'\in\Gamma_F$.}
		\end{equation}
		 We assume that each curve $\gamma$ in these two curve families is parameterized by arc length.
		 Set $\Gamma^j_E=\{\gamma\in \Gamma_E:\ u(\gamma(t))\ge \frac {3}{2} \text{\rm \ for all\ \ } t\geq \lambda 2^{j}\}.$ Then $\Gamma_E=\bigcup_j \Gamma_E^j$ and hence, by the subadditivity of the modulus,
		 $$\text{\rm Mod}_p(\Gamma_E)\le \sum_j \text{\rm Mod}_p(\Gamma_E^j).$$ Hence \eqref{eq3.10} gives the existence of $j_E$ and $\delta_E>0$ so that 
		 $\text{\rm Mod}_p(\Gamma_E^{j_E})\ge \delta_E.$ By arguing analogously for $F,$ we find $j_F$ so that $\text{\rm Mod}_p(\Gamma_F^{j_F})\ge \delta_F>0$ and 
		 $u(\gamma(t))\le \frac {1}{2}$ when $\gamma\in \Gamma_F^{j_F}$ and $t\ge \lambda 2^{j_F}.$		 
		 Let $j\ge \max\{j_E,j_F\}.$
We define sets $E_j$, $F_j$ by setting 
		\[E_j:= \left\{ A_{\lambda 2^{j+1}}\bigcap \left (\bigcup_{\gamma\in \Gamma_E^{j_E}}\gamma\right )\right\}
		\]  
	and 
		\[ F_j:=\left\{ A_{\lambda 2^{j+1}}\bigcap \left (\bigcup_{\gamma \in \Gamma_F^{j_F}} \gamma\right )\right\}.
		\] 
		Then $u(x)\geq 3/2$ for all $x\in E_j$ and $u(x)\leq 1/2$ for all $x\in F_j$. Moreover, $(\lambda2^{j+1})^p\text{\rm Mod}_p(\Gamma_{E}^{j_E})\leq \mu(E_j)$ and $(\lambda2^{j+1})^p\text{\rm Mod}_p(\Gamma_{F}^{j_F})\leq \mu(F_j)$  since every curve in these families has a subcurve of length no less than $\lambda 2^{j+1}$ in $A_{\lambda 2^{j+1}}$ and hence $\chi_{E_j}/(\lambda2^{j+1})$  and $\chi_{F_j}/(\lambda2^{j+1})$  are admissible  functions for computing $\text{\rm Mod}_p(\Gamma_{E}^{j_E})$  and $\text{\rm Mod}_p(\Gamma_{F}^{j_F}),$ respectively.

	Notice that $1\leq |u(x)-u(y)| \leq |u(x)-u_{B(x,2^j/c_1)}|+|u_{B(x,2^j/c_1)}-u_{B(y,2^j/c_1)}|+|u_{B(y,2^j/c_1)}-u(y)|$ for all $x\in E_j, y\in F_j$ where $c_1$ is as in Definition \ref{dyadic}. We will consider three cases corresponding to Lemma \ref{lem3.1}-\ref{lem3.2}.
	Applying Lemma \ref{lem3.1} for the pair $(E_j, F_j)$,
	 the estimate \eqref{eq3.2} with $1\leq |u(x)-u(y)|$ for all $x\in E_j, y\in F_j$ gives
		\[\mu(c_2\cdot A_{\lambda 2^{j+1}})\lesssim (2^j)^p\int_{c_2\cdot A_{\lambda 2^{j+1}}}\rho_u^pd\mu 
		\]if $|u(x)-u_{B(x,2^j/c_1)}|\leq \frac{1}{5}$ and $|u(y)-u_{B(y,2^j/c_1)}|\leq \frac{1}{5}$ hold for some $x\in E_j, y\in F_j$. Applying Lemma \ref{lem3.2} for each $E_j,F_j$, the estimate \eqref{eq3.5}  gives 
		\[\min\{\mu(E_j),\mu(F_j)\}\lesssim (2^j)^p\int_{c_2\cdot A_{\lambda 2^{j+1}}}\rho_u^pd\mu 
		\]if either $|u(x)-u_{B(x,2^j/c_1)}|> \frac{1}{5}$ holds for every $x\in E_j$ or $|u(y)-u_{B(y,2^j/c_1)}|>\frac{1}{5}$ holds for every $y\in F_j$. Since both $E_j$  and $F_j$ are subsets of $c_2\cdot A_{\lambda 2^{j+1}}$, the above estimates imply that 
			\[ \min\{\mu(E_j),\mu(F_j)\}\lesssim (2^j)^p\int_{c_2\cdot A_{\lambda 2^{j+1}}}\rho_u^pd\mu.
			\]
	From our upper estimates on $\text{\rm Mod}_p(\Gamma_{E}^{j_E})$ and on $\text{\rm Mod}_p(\Gamma_{F}^{j_F})$ from above we conclude that
	\[\min \{\text{\rm Mod}_p(\Gamma_{E}^{j_E}), \text{\rm Mod}_p(\Gamma_{F}^{j_F})\}\lesssim \int_{c_2\cdot A_{\lambda 2^{j+1}}}\rho_u^pd\mu.
		\]
	By inserting the strictly positive lower bounds for these two modulus, we conclude that
	\[0<\min\{\delta_E,\delta_F\}\lesssim \int_{c_2\cdot A_{\lambda 2^{j+1}}}\rho_u^pd\mu.
	\]
	Since $u\in \dot N^{1,p}(X),$ a contradiction follows by letting $j$ tend to infinity. 	
\end{proof}

\section{Proofs of Theorems \ref{thm1.2-1802}-\ref{thm4}}\label{sec4}

We begin with a sequence of auxiliary lemmas whose proofs rely on arguments similar to those in \cite{EKN,KNW,K}. Fix $1\leq p<\infty$.

\begin{lemma}\label{lem4.1}  Let $(X,d,\mu)$ be a doubling metric measure space that supports a $p$-Poincar\'e inequality. Assume that $X$ has the annular chain property at O. Suppose that $X$ has a weak polar coordinate system at $O$ as in \eqref{intro-polar}.
If $R_{p}(h,O)<\infty$, then the existence  and uniqueness of \eqref{infinity-limit} hold for
 $\sigma$-a.e $\xi \in \mathbb{S}$.
\end{lemma}
\begin{proof} Let $u\in\dot N^{1,p}(X)$.
By Theorem \ref{thm1},  there exists $c\in\mathbb R$ such that 
\begin{equation}\label{eq4.1-1205}
\lim_{t\to\infty}u(\gamma(t)) \text{\rm \ \ exists and \ \ }\lim_{t\to\infty} u(\gamma(t))=c \text{\rm \ \ for $p$-a.e $\gamma\in\Gamma^{+\infty}$.}
\end{equation}
Let $F$ be the collection of all $\xi\in \mathbb S$ such that $\lim_{t\to\infty}u(\gamma_\xi^O)$ does not exist or exists but is  not equal to $c$.  It suffices to prove that 
\begin{equation} \label{eq4.2-1205} \sigma(F)=0.
\end{equation}
 We first set $F_{\rm bad}:=\{\xi\in F: \liminf_{t\to\infty} d(O, \gamma_\xi^O)\leq 1 \}$. Let $\gamma_\xi^O\in F_{\rm bad}$. By the definition of infinite curves, we have $\gamma_\xi^O\setminus B(O,r)\neq \emptyset$ for all  $r>0$ and hence there is a sequence $\{t_n\}_{n\in\mathbb N}$ such that $d(O, \gamma_\xi^O(t_n))\geq n$. From $\liminf_{t\to\infty}d(O,\gamma_\xi^O)\leq 1$, it follows that there is also (after passing to a subsequence) 
a sequence $\{s_n\}_{n\in\mathbb N}$ with $\lim_{n\to \infty}s_n=\infty$ such that $d(O,\gamma_\xi^O(s_n))\leq 1$ and $s_n\in (t_{n},t_{n+1})$ for all $n\in\mathbb N$. Roughly speaking, $\gamma_\xi^O$ 
oscillates infinitely often between the unit ball and far away parts of our space. Therefore, it is clear that 
\[
\int_{\gamma_\xi^O}\chi_{B(O,2)\setminus B(O,1)} ds =\infty
\]
for each $\gamma_\xi^O\in F_{\rm bad}$. Also, $h(x)>0$ almost everywhere in the complement of $B(O,1)$ because $R_p(h,O)<\infty.$
We then obtain that for all $m\in\mathbb N$,
\begin{align}
m\ \sigma(F_{\rm bad})\leq & \int_{F_{\rm bad}}\int_{\gamma_\xi^O}\chi_{B(O,2)\setminus B(O,1)}dsd\sigma(\xi)\notag \\
\leq & \int_{\mathbb S} \int_{\gamma_\xi^O} \chi_{B(O,2)\setminus B(O,1)}dsd\sigma(\xi)\notag \\
\leq & \mathcal C \int_{X} \frac{\chi_{B(O,2)\setminus B(O,1)}}{h}d\mu\notag \\
\leq & \mathcal C  [\mu(B(O,2)\setminus B(O,1))]^{1/p} \left(\int_{B(O,2)\setminus B(O,1)}h^{\frac{p}{1-p}}d\mu \right)^{\frac{p-1}{p}}\notag \\
\leq & \mathcal C [\mu(B(O,2)\setminus B(O,1))]^{1/p} R_p^{\frac{p-1}{p}}(h,O)\label{eq4.3-1205}
\end{align}
if $p>1$, and similarly 
\begin{equation}\label{eq4.4-1205}
m\ \sigma(F_{\rm bad})\leq \mathcal C \mu(B(O,2)\setminus  B(O,1)) \ R_{1}(h,O)
\end{equation}  by  \eqref{intro-polar} and the H\"older inequality, and where $\mathcal C>0$ is the constant of our weak polar coordinate system \eqref{intro-polar}. Here the weak polar coordinate system \eqref{intro-polar} can be applied since $\frac{\chi_{B(O,2)\setminus B(O,1)}}{h}$ is integrable. Letting $m\to\infty$, the above estimates give 
\begin{equation}\label{eq4.5-1705}
\sigma(F_{\rm bad})=0
\end{equation}
 for $1\leq p<\infty$ since the right-hand side of \eqref{eq4.3-1205}-\eqref{eq4.4-1205} is 
bounded. 
 
 Let $\gamma_\xi^O\in F\setminus F_{\rm bad}$. Then $\liminf_{t\to \infty}d(O,\gamma_\xi^O(t))>1$, and so there is $t_\xi>0$ such that $d(O,\gamma_\xi^O(t_\xi))=1$ 
and $\gamma^O_\xi(t)\cap B(O,1)=\emptyset$ for all $t> t_\xi$. Let us define $\hat \gamma_\xi^O$, a subcurve of $\gamma_\xi^O,$ by setting $\hat \gamma_\xi^O(t)=\gamma_\xi^O(t+t_\xi)$ for all $t\geq 0$. Let $\hat\Gamma^O(F)$ be the 
collection of all these infinite subcurves $\hat\gamma_\xi^O$ of $\gamma_\xi^O$ with 
respect to $\xi\in F$ satisfying $\liminf_{t\to\infty}d(O,\gamma_\xi^O(t))>1$.
Then, for all $\hat\gamma_\xi^O\in\hat\Gamma^O(F)$,  
  \begin{equation}\label{eq4.7-1705}
  \hat \gamma_\xi^O\cap B(O,1)=\emptyset
  \end{equation}
   and 
\[
\lim_{t\to\infty}u(\gamma_\xi^O(t)) \text{\rm \ \ does not exist or is not equal to $c$}
\]
since $\xi \in F$. By \eqref{eq4.1-1205}, we also have 
\begin{equation}\label{eq4.8-1705}
\text{\rm Mod}_p (\hat\Gamma^O(F))=0.
\end{equation}
   Let $g$ be admissible for computing  $\text{\rm Mod}_p( \hat\Gamma^O(F))$. Then we may assume that $g=0$ on $B(O,1)$ since $\hat \gamma_\xi^O\cap B(O,1)=\emptyset$ by \eqref{eq4.7-1705}. Suppose first that $p>1.$
By the same arguments as for $\eqref{eq4.3-1205}$-$\eqref{eq4.4-1205}$, we obtain that 
\[
\sigma(F\setminus F_{\rm bad})\leq \mathcal C \left( \int_X g^pd\mu \right)^{1/p}  R_p^{\frac{p-1}{p}}(h,O).
\]
Since $g$ is arbitrary, it follows that 
\[
\sigma(F\setminus F_{\rm bad})\leq \mathcal C \left( \text{\rm Mod}_p(\hat \Gamma^O(F))\right)^{1/p}  R_p^{\frac{p-1}{p}}(h,O).
\]
Combining with \eqref{eq4.5-1705}, we conclude that 
\begin{equation}\label{equation4.8-1705}
	\sigma(F)\leq \mathcal C \left( \text{\rm Mod}_p(\hat \Gamma^O(F))\right)^{1/p} R_p^{\frac{p-1}{p}}(h,O).
\end{equation}
As $R_p(h,O)<\infty$ and $\text{\rm Mod}_p(\hat\Gamma^O(F))=0$ by \eqref{eq4.8-1705},  we conclude that $\eqref{eq4.2-1205}$ holds. The case $p=1$ follows via an analogous argument.
\end{proof}


\begin{lemma}
	\label{lem4.4} 
	If $\mathcal R_{p}(O)=\infty$, then there is $u\in\dot N^{1,p}(X)$ such that 
		$\liminf_{t\to\infty}u(\gamma(t))= \infty 
		$
		 for every $\gamma\in\Gamma^\infty$.
\end{lemma}
\begin{proof}
	By $\mathcal R_{p}(O)=\infty$,
	there exists a  sequence $\{n_{k}\}_{k=1}^\infty$ (or $\{i_{k}\}_{k=1}^\infty$) with
\begin{equation}\label{equ4.6}
\text{ $n_1<n_2<\ldots$ (or $i_1<i_2<\ldots$) such that $\lim_{k\to\infty}n_k=\infty$ (or $\lim_{k\to\infty}i_k=\infty$),}
\end{equation}	
	
	\begin{equation}
	\label{eq4.6} \sum_{i=n_k}^{n_{k+1}} (2^i)^{\frac{p}{p-1}}\mu^{\frac{1}{1-p}}(A_{2^i}) >2^k \text{\rm  \ \ \, if $p>1$ $\big($or  \ \ \ }2^{i_k}\mu^{-1}(A_{2^{i_k}})>2^k \text{\ \ \ \rm if  $p=1\big)$}.
	\end{equation} 
	Here $A_{2^i}:=B(O,2^{i+1})\setminus B(O,2^i)$. Let
	\begin{equation}\notag
	\label{eq4.7}  g_p(x)=\sum_{k=1}^\infty \left (\sum_{i=n_{k}}^{{n_{k+1}}}\frac{(2^i)^{\frac{1}{p-1}}\mu^{\frac{1}{1-p}}(A_{2^i})}{\sum_{i=n_k}^{n_{k+1}} (2^i)^{\frac{p}{p-1}}\mu^{\frac{1}{1-p}}(A_{2^i}) } \chi_{A_{2^i}}(x)\right ) \text{\rm \  if $p>1$\ }
	\big({\rm or\ \ }g_1(x)=\sum_{k=1}^\infty 2^{-i_k}\chi_{A_{2^{i_k}}}(x)\big).
	\end{equation}
	We define
	$u(x):=\inf \int_{\gamma_{O,x}} g_pds$  for $x\in X$
	where the infimum is taken over all rectifiable curves $\gamma_{O,x}$ connecting $O$ and $x$. Then $g_p$ is an upper gradient of $u$, see for instance \cite[Page 188-189]{pekka}.
	We have the fact (see for instance \cite[Proposition 5.1.11]{pekka}) that
	$\int_{\gamma_{O,x}\cap A_{2^i}}ds \geq \text{\rm diam}(\gamma_{O,x}\cap A_{2^i})\gtrsim 2^i 
	$ for every $\gamma_{O,x}$ with $d(O,x)\geq 2^{i+1}$.
	Here $\text{\rm diam}(\gamma_{O,x}\cap A_{2^i})$ is the diameter of $\gamma_{O,x}\cap A_{2^i}$.
	We let $\gamma\in\Gamma^\infty$ and let $N>1$. For all $x\in  \gamma$ with $d(O,x)=N$, we have that 
	\begin{align*}
	u(x)=&\inf_{\gamma_{O,x}} \int_{\gamma_{O,x}} g_pds =  \inf_{\gamma_{O,x}} \sum_{2^{n_{k+1}}\leq N}\sum_{i=n_k}^{n_{k+1}} \frac{(2^i)^{\frac{1}{p-1}}\mu^{\frac{1}{1-p}}(A_{2^i})}{\sum_{i=n_k}^{n_{k+1}} (2^i)^{\frac{p}{p-1}}\mu^{\frac{1}{1-p}}(A_{2^i}) } \int_{\gamma_{O,x}\cap A_{2^i}}ds \\
	\gtrsim&  \sum_{2^{n_{k+1}}\leq N}\sum_{i=n_k}^{n_{k+1}} \frac{(2^i)^{\frac{1}{p-1}}\mu^{\frac{1}{1-p}}(A_{2^i})}{\sum_{i=n_k}^{n_{k+1}} (2^i)^{\frac{p}{p-1}}\mu^{\frac{1}{1-p}}(A_{2^i}) } 2^i
	=\sum_{2^{n_{k+1}}\leq N}1 \to \infty \text{\rm \ \ as $N\to \infty$\ \ \ if $p>1$,}
	\end{align*}
	(or that
	$
	u(x)=\inf_{\gamma_{O,x}}\int_{\gamma_{O,x}}g_1ds = \inf_{\gamma_{O,x}} \sum_{2^{i_k}\leq  N} 2^{-i_k}\int_{\gamma_{O,x}\cap A_{2^{i_k}}}ds 
	\gtrsim  \sum_{2^{i_k}\leq N}2^{-i_k} 2^{i_k}= \sum_{2^{i_k}\leq  N}1\to\infty 
	$ as $N\to\infty$).
	Hence
	$\liminf_{t\to \infty}u(\gamma(t))=\infty 
	$
	 for every $\gamma\in\Gamma^\infty$.
	It remains to show that $g_p$ is $p$-integrable. 
	Using \eqref{equ4.6}-\eqref{eq4.6}, we have that 
	\begin{align*}
	 \int_{X}g_p^pd\mu =\sum_{k=1}^\infty \sum_{i=n_{k}}^{{n_{k+1}}}  \int_{A_{2^i}}\left (\frac{(2^i)^{\frac{1}{p-1}}\mu^{\frac{1}{1-p}}(A_{2^i})}{\sum_{i=n_k}^{n_{k+1}} (2^i)^{\frac{p}{p-1}}\mu^{\frac{1}{1-p}}(A_{2^i})}\right )^p d\mu\notag 
	= \sum_{k=1}^\infty \frac{1}{ \left (\sum_{i=n_k}^{n_{k+1}} (2^i)^{\frac{p}{p-1}}\mu^{\frac{1}{1-p}}(A_{2^i}) \right )^{p-1}} \leq \sum_{k=1}^\infty\frac{1}{2^{k(p-1)}} 
	\end{align*}
	if $p>1$ (or that
	$
	 \int_{X}g_1d\mu =\sum_{k=1}^\infty \int_{A_{2^{i_k}}} 2^{-i_k}d\mu =\sum_{k=1}2^{-i_{k}}\mu(A_{2^{i_k}})\leq \sum_{k=1}^\infty\frac{1}{2^{k}}).\notag
	$ The claim follows.
	\end{proof}

\begin{example}
	\label{lem4.3}  If $h(x)=|x-O|^{n-1}w(x)\chi_{\mathbb R^n\setminus B(O,1)}(x)$, where $w$ is a classical Muckenhoupt $\mathcal A_p$-weight on $\mathbb R^n$ and where $n\geq 2$, then $ R_{p}(h,O)\approx \mathcal R_{p}(O)$.
\end{example}
\begin{proof}Let $A_{2^i}=B(O,2^{i+1})\setminus B(O,2^{i})$.
	Notice that
	\[R_{p}(h,O) =\sum_{i=0}^\infty \int_{A_{2^i}}|x-O|^{\frac{p(n-1)}{1-p}}w^{\frac{1}{1-p}}(x)dx \text{ if $p>1$ $\big($ or\ } R_{1}(h,O)=\sup_{i\in\mathbb N}\left \||x-O|^{1-n}w^{-1}(x)\right \|_{L^\infty(A_{2^i})}\big).
	\] 
	Since $|x-O|\approx 2^i$ for any $x\in A_{2^i}$ and $|A_{2^i}|\approx (2^i)^n$, it follows that
	\[R_{p}(h,O)\approx \sum_{i=0}^\infty (2^i)^{\frac{p(n-1)}{1-p}+n}\dashint_{A_{2^i}}w^{\frac{1}{1-p}}(x)dx\text{\ \ \rm if $p>1$ $\big($ or  \ }
	R_{1}(h,O)\approx \sup_{i\in\mathbb N}\left (2^i)^{1-n}\|w^{-1}(x)\right \|_{L^\infty(A_{2^i})}\text{\ \ \rm if $p=1$}\big).
	\]
	Because $w$ is a Muckenhoupt $\mathcal A_p$-weight, we have from \eqref{eq1.6-0610}-\eqref{eq1.7-0610} that 
	\[\left (\dashint_{A_{2^i}}w^{\frac{1}{1-p}}(x)dx\right ) \approx \left (\dashint_{A_{2^i}}w(x)dx\right )^{\frac{1}{1-p}} \approx (2^i)^{\frac{-n}{1-p}} \mu^{\frac{1}{1-p}}(A_{2^i}) \text{ if }p>1,
	\]
	$\big($or that
	$\left \| w^{-1}(x)\right \|_{L^{\infty}(A_{2^i})} \approx \left (\dashint_{A_{2^i}}w(x)dx\right )^{-1}\approx(2^i)^n\mu^{-1}(A_{2^i})\big).
	$ Inserting these into the above formula of $R_{p}(h,O)$, we obtain the claim.
\end{proof}

\begin{lemma}
\label{lem-ii-to-iii}  Let $(X,d,\mu)$ be a doubling metric measure space that supports a $p$-Poincar\'e inequality. Assume that $X$ is complete. If $\mathcal R_p(O)<\infty$ then $\text{\rm Mod}_p(\Gamma^\infty)>0$.
\end{lemma}
\begin{proof}
Since $X$ is complete and doubling, and supports a $p$-Poincar\'e inequality, there exists a geodesic metric $\hat{d}$ that is biLipschitz equivalent to $d$, see for instance in \cite[Corollary 8.3.16]{pekka}. Let $\mathcal R_p(\hat{d},O)$ be the version of $\mathcal R_p(O)$ in $(X,\hat{d},\mu)$ and $B_{\hat{d}}(x,r)$ be the ball with radius $r$ and center at $x$ on $(X,\hat{d})$. It follows that $\mathcal R_p(\hat{d},O)<\infty$ since $\mathcal R_p(O)<\infty$. 
Since $\hat{d}$ is geodesic  we are allowed to employ
Theorem 2.10 in \cite{HoKo01}. This result gives us a constant $C>0$ such that for all $i\in\mathbb N$
\[
C\leq {\rm Cap}_p(B_{\hat{d}}(O,1),B_{\hat{d}}(O,2^i)) \left(\mathcal R_p(\hat{d},O)\right)^{p-1}
\]
 when $p>1$  $\big($or otherwise
 \[
 C\leq {\rm Cap}_1(B_{\hat{d}}(O,1),B_{\hat{d}}(O,2^i)) \mathcal R_1(\hat{d},O)\big).
\] 
Here ${\rm Cap}_p(B_{\hat{d}}(O,1),B_{\hat{d}}(O,2^i))$ is the quantity
\[
{\rm Cap}_p(B_{\hat{d}}(O,1),B_{\hat{d}}(O,2^i)):=\inf\int_Xg_u^pd\mu
\]where the infimum is taken over all functions $u:X\to\mathbb R$ with the minimal upper gradient $g_u$ such that $u|_{B_{\hat{d}}(O,1)}\equiv 1$ and $u|_{X\setminus B_{\hat{d}}(O,2^i)}\equiv 0$.
 It follows from $\mathcal R_p(\hat{d},O)<\infty$  that  ${\rm Cap}_p(B_{\hat{d}}(O,1),B_{\hat{d}}(O,2^i))\geq \delta>0$  for all $i\in\mathbb N$, and for some $\delta$ only depending on $C$ and $\mathcal R_p(\hat{d},O)$. 
By \cite[Page 12]{HK98}, we obtain that \[{\rm Cap}_p(B_{\hat{d}}(O,1),B_{\hat{d}}(O,2^i))\leq \text{\rm Mod}_p (B_{\hat{d}}(O,1),B_{\hat{d}}(O,2^i)) \text{\rm\ \ for all $i\in\mathbb N$.} \] Here $\text{\rm Mod}_p (B_{\hat{d}}(O,1),B_{\hat{d}}(O,2^i))$ is the $p$-modulus of the family of all retifiable curves connecting $B_{\hat{d}}(O,1)$ and $X\setminus B_{\hat{d}}(O,2^i)$. Then $ \text{\rm Mod}_p (B_{\hat{d}}(O,1),B_{\hat{d}}(O,2^i))\geq \delta >0 \text{\rm \ \ for all $i\in\mathbb N$}.$
It follows that $X$ is $p$-hyperbolic in the sense of \cite[Definition 2.4]{Sha19}. Hence \cite[Theorem 4.2]{Sha19} yields that $\text{\rm Mod}_p(\Gamma^\infty)>0$.
The claim follows. 
\end{proof}

\begin{proof}[Proof of Theorem \ref{thm1.2-1802}]

${\rm I.}\Rightarrow {\rm II.}$ is given by Lemma \ref{lem-ii-to-iii}.

${\rm II.}\Rightarrow {\rm III.}$ is given by Theorem \ref{thm1}.

${\rm III.}\Rightarrow {\rm I.}$ is given by Lemma \ref{lem4.4}.
\end{proof}

\begin{proof}[Proof of Theorem \ref{thm4}]
	
	$R_{p}(h,O)<\infty \Rightarrow 3.$ is given by Lemma \ref{lem4.1}.
	
	$3. \Rightarrow 2. \Rightarrow 1.$ is trivial.

$1. \Rightarrow \mathcal R_{p}(O)<\infty$ is given by Lemma \ref{lem4.4}. 
	
	$R_{p}(h,O)<\infty \Rightarrow 4.$ is given by the  estimate \eqref{equation4.8-1705}.

	$R_p(h,O)<\infty\Rightarrow 5.$ is given by the  estimate \eqref{equation4.8-1705}.
	
	$4. \Rightarrow \mathcal R_{p}(O)<\infty$ is given by Lemma \ref{lem4.4} together with Theorem \ref{thm1}.

	$5. \Rightarrow \mathcal R_p(O)<\infty$ is given by Lemma \ref{lem4.4}
together with Theorem 	\ref{thm1}. Indeed, if $\mathcal R_p(O)=\infty$ then $\text{\rm Mod}_p(\hat\Gamma^{O}(F))=0$ for all subsets $F\subseteq\mathbb S$ with $\sigma(F)>0$.
\end{proof}
\section*{Acknowledgements}
We thank the anonymous referee for a careful reading and thoughtful
comments on the paper.
Both authors have been supported by the Academy of Finland Grant number 323960.

\end{document}